\documentclass[11pt,twoside]{article}

\usepackage[utf8]{inputenc} %
\usepackage[T1]{fontenc}    %
\usepackage[english]{babel} %
\usepackage{palatino}        %
\usepackage[a4paper,margin=1in,marginpar=1in]{geometry} %
\usepackage{xargs}          %
\usepackage{xparse}         %
\usepackage{csquotes}       %
\usepackage{quoting}       %

\usepackage[shortlabels]{enumitem}

\usepackage[nottoc,numbib]{tocbibind} %

\usepackage{amssymb, amsmath, amsthm, amsfonts}
\usepackage{esint}
\usepackage{mathtools}

\usepackage[backend=biber,giveninits=true,minbibnames=5,maxbibnames=5]{biblatex}
\usepackage{xcolor}
\usepackage[pdfpagelabels]{hyperref}
\definecolor{green}{HTML}{2ECC71}
\definecolor{blue}{HTML}{3498DB}
\definecolor{red}{HTML}{E74C3C}
\hypersetup{%
colorlinks,
linktocpage,
urlcolor  = blue,
citecolor = green,
linkcolor = red}
\usepackage{listings}   %
\definecolor{lstgreen}{rgb}{0,0.5,0}
\usepackage[capitalise,nameinlink,sort]{cleveref} %
\AfterEndEnvironment{proof}{\noindent\ignorespaces} %
\makeatletter
\def\@endtheorem{\endtrivlist}
\makeatother
\Crefname{equation}{}{}
\newtheorem{theorem}{Theorem}
\Crefname{theorem}{Theorem}{Theorems}
\newtheorem{lemma}[theorem]{Lemma}
\Crefname{lemma}{Lemma}{Lemmas}
\newtheorem{proposition}[theorem]{Proposition}
\Crefname{proposition}{Proposition}{Propositions}
\newtheorem{remark}{Remark}
\Crefname{remark}{Remark}{Remarks}
\usepackage{calrsfs}
\DeclareMathAlphabet{\pazocal}{OMS}{zplm}{m}{n}

\newcommand{\CA}{\mathcal{A}}
\newcommand{\CC}{\mathcal{C}}
\newcommand{\CL}{\mathcal{L}}
\newcommand{\CM}{\mathcal{M}}
\newcommand{\cE}{\pazocal{E}}
\newcommand{\gM}{\mathfrak{M}}
\newcommand{\gZ}{\mathfrak{Z}}
\newcommand{\dN}{\mathbb{N}}
\newcommand{\dR}{\mathbb{R}}
\newcommand{\dd}{\mathrm{d}}
\DeclareMathOperator{\e}{e}  %
\DeclareMathOperator{\esp}{\mathbf{E}}

\DeclareMathOperator{\dom}{\mathbb{D}om}
\DeclareMathOperator{\Gammaa}{\mathbf{\Gamma}}

\author{Ronan Herry}
\AtEndDocument{\bigskip\footnotesize%
  R.H.\ is a member of \textsc{Institut für Angewandte Mathematik, Bonn Universität}
  \par
  Endenicher Allee 60, D-53115, Bonn, Germany.
  \par
  \href{mailto:herry@iam.uni-bonn.de}{\texttt{ronan.herry@live.fr}}
  \par
  ORCid: \href{https://orcid.org/0000-0001-6313-1372}{0000-0001-6313-1372}
}

\DeclareSourcemap{
  \maps[datatype=bibtex]{
    \map{
      \step[fieldsource=doi,final]
      \step[fieldset=url,null]
    }  
  }
}

\begin{document}
\title{A note on the carré du champ on the Poisson space}

\maketitle
\begin{abstract}
The goal of this short note is to establish, in complete generality, the representation for the carré du champ operator associated with the Ornstein-Uhlenbeck semi-group on the Poisson space in terms of the add-one and drop-one operators (see \cref{proposition:carre_du_champ} below).
\end{abstract}

\subsubsection*{Poisson setting}
We fix $(Z, \gZ)$ a measurable space equipped with a $\sigma$-finite measure $\nu$.
In particular, we do not make any topological assumptions on $(Z, \gZ)$.
We consider $\CM$ be the space of all countable sums of $\dN$-valued measures on $(Z, \gZ)$.
The space $\CM$ is endowed with the $\sigma$-algebra $\gM$, generated by the \emph{cylindrical mappings}
\begin{equation*}
  \xi \in \CM \mapsto \xi(B) \in \dN \cup \{\infty\},\quad B \in \gZ.
\end{equation*}
The Poisson point process with intensity $\nu$ is the only probability $\Pi$ on $\CM$ such that the \emph{Mecke equation} holds:
\begin{equation}\label{equation:Mecke}
  \int \int u(\eta, z) \eta(\dd z) \Pi(\dd \eta) = \int \int u(\eta + \delta_{z}, z) \Pi(\dd \eta) \nu(\dd z),
\end{equation}
for all measurable $u \colon \CM \times Z \to [0,\infty]$.
Poisson processes with $\sigma$-finite intensity exist \cite[Theorem 3.6]{LastPenrose}.
Note that, if, in the previous equation, $f$ is replaced by a measurable function with values in $\dR$, the previous formula still holds provided both sides of the identity are finite when we replace $f$ by $|f|$.
Integration with respect to $\Pi$ will also be denoted by the probabilistic notation $\esp_{\Pi}$.

\subsubsection*{The add and drop operators}
Given $z \in Z$ and $F \colon \CM \to \dR$ measurable, we let
\begin{align*}
  & D^{+}_{z} F(\eta) =  F(\eta + \delta_{z}) - F(\eta); \\
  & D^{-}_{z} F(\eta) =  (F(\eta) - F(\eta - \delta_{z})) 1_{z \in \eta}.
\end{align*}
The operator $D^{+}$ (resp.\ $D^{-}$) is called the \emph{add operator} (resp.\ \emph{drop operator}).
Due to the Mecke formula \cref{equation:Mecke}, these operations do not depend on the choice of the representative of $f$ $\Pi$-almost surely.
\begin{lemma}\label{lemma:D_bounded}
  Let $F \in \CL^{\infty}(\Pi)$, then $D^{+}F \in \CL^{\infty}(\Pi \otimes \nu)$.
\end{lemma}
\begin{proof}
  First of all, $\delta \colon Z \ni z \mapsto \delta_{z} \in \CM$ is measurable (if $A$ is of the form $\{ \eta(B) = k \}$ for some $B \in \gZ$, then the pre-image by $\delta$ of $A$ is $B$, if $k = 1$; and the pre-image is empty, if $k > 1$).
  Hence, $D^{+}F$ is bi-measurable.
  Now let
  \begin{align*}
    & U = \{t \in \dR,\, \text{such that}\ \Pi(F \geq t) = 0 \}; \\
    & V = \{t \in \dR,\, \text{such that}\ (\Pi \otimes \nu)(F + D_{z}^{+}F \geq t) = 0 \}.
  \end{align*}
  By assumption $U \ne \emptyset$, and we want to show that $V \ne \emptyset$.
  Take $t \in U$, by the Mecke formula \cref{equation:Mecke}, we have that
  \begin{equation*}
    \int \int 1_{\{F + D_{z}^{+}F \geq t\}} \nu(\dd z) \Pi(\dd \eta) = \int \int 1_{\{F \geq t\}} \eta(\dd z) \Pi(\dd \eta) = 0.
  \end{equation*}
  Hence $t \in V$, this concludes the proof.
\end{proof}

\subsubsection*{Malliavin derivative}
For a random variable $F$, we write $F \in \dom D$ whenever: $F \in \CL^{2}(\Pi)$ and
\begin{equation*}
  \int_{Z} \int {(D_{z}^{+}F(\eta))}^{2} \Pi(\dd \eta) \nu(\dd z) < \infty.
\end{equation*}
Given $F \in \dom D$, we write $DF$ to denote the random mapping $DF \colon Z \ni z \mapsto D_{z}^{+}F$.
We regard $D$ as an unbounded operator $\CL^{2}(\Pi) \to \CL^{2}(\Pi \otimes \nu)$ with domain $\dom D$.
The operator $D$ is closed \cite[Lemma 3]{LastAnaSto} and thus $\dom D$ is Hilbert when equipped with the scalar product
\begin{equation*}
  (F, G) \mapsto \Pi(FG) + (\Pi \otimes \nu)(DF DG).
\end{equation*}

\subsubsection*{The divergence operator}
We consider the \emph{divergence operator} $\delta = D^{*} \colon \CL^{2}(\Pi \otimes \nu) \to \CL^{2}(\nu)$, that is the unbounded adjoint of $D$.
Its domain $\dom \delta$ is composed of random functions $u \in \CL^{2}(\Pi \otimes \nu)$ such that there exists a constant $c > 0$ such that
\begin{equation*}
  \left|\int \int D^{+}_{z}F(\eta) u(\eta, z) \nu(\dd z) \Pi(\dd \eta) \right| \leq c \sqrt{\Pi(F^{2})},\quad \forall F \in \dom D.
\end{equation*}
For $u \in \dom \delta$, the quantity $\delta u \in \CL^{2}(\Pi)$ is completely characterised by the duality relation
\begin{equation}\label{equation:integration_by_parts_delta}
  \esp_{\Pi} G \delta u = \int \int u(\eta, z) D_{z} F(\eta) \Pi(\dd \eta) \nu(\dd z), \quad \forall F \in \dom D.
\end{equation}
From \cite[Theorem 5]{LastAnaSto}, we have the following Skorokhod isometry.
For $u \in \CL^{2}(\Pi \otimes \nu)$, $u \in \dom \delta$ if and only if $\int {(D_{z}^{+}u(\eta, z'))}^{2} \Pi(\dd \eta) \nu(\dd z) \nu(\dd z') < \infty$ and, in that case:
\begin{equation}\label{equation:Skorokhod_isometry}
  \esp_{\Pi} {(\delta u)}^{2} = \int \int {u(\eta, z)}^{2} \nu(\dd z) \Pi(\dd \eta) + \int \int D_{z}^{+} u(\eta, z') D_{z'}^{+} u(\eta, z) \Pi(\dd\eta) \nu(\dd z) \nu(\dd z').
\end{equation}
The Skorokhod isometry implies the following Heisenberg commutation relation.
For all $u \in \dom \delta$, and all $z \in Z$ such that $z' \mapsto D_{z}^{+}u(z') \in \dom \delta$:
\begin{equation*}
  D_{z}\delta u(\eta) = u(\eta, z) + \delta D_{z}^{+} u(\eta, \cdot).
\end{equation*}
From \cite[Theorem 6]{LastAnaSto}, we have the following pathwise representation of the divergence: if $u \in \dom \delta \cap \CL^{1}(\Pi \otimes \nu)$, then
\begin{equation}\label{equation:divergence_pathwise}
  \delta u(\eta) = \int u(\eta, z) \eta(\dd z) - \int u(\eta, z) \nu(\dd z).
\end{equation}
Note that $\dom \delta \cap \CL^{1}(\Pi \otimes \nu)$ is dense in $\dom \delta$.

\subsubsection*{The Ornstein-Uhlenbeck generator}
The \emph{Ornstein-Uhlenbeck generator} $L$ is the unbounded self-adjoint operator on $\CL^{2}(\Pi)$ verifying
\begin{equation*}
  \dom L = \{ F \in \dom D,\, \text{such that}\ DF \in \dom \delta \} \quad \text{and} \quad L = - \delta D.
\end{equation*}
Classically, $\dom L$ is endowed with the Hilbert norm $\esp_{\Pi} \left(F^{2} + {(LF)}^{2}\right)$.
The eigenvalues of $L$ are the non-positive integers and for $q \in \dN$ the eigenvectors associated to $-q$ are the so-called iterated Poisson stochastic integrals of order $q$ (see \cite{LastAnaSto} for details).
The kernel of $L$ coincides with the set of constants and the \emph{pseudo-inverse of $L$} is defined on the quotient $\CL^{2}(\Pi) \setminus \ker L$, that is the space of centered square integrable random variables.
For $F \in \CL^{2}(\Pi)$ with $\Pi(F) = 0$, we have $LL^{-1}F = F$.
Moreover, if $F \in \dom L$, we have $L^{-1}LF = F$.
As a consequence of \cref{equation:Skorokhod_isometry}, $\dom D^{2} = \dom L$.

\subsubsection*{The Dirichlet form}
We refer to \cite[Chapter 1]{BouleauHirsch} for more details about the formalism of Dirichlet forms.
The introduction of \cite{AGSBakryEmery} also provides an overview of the subject.
For every $F, G \in \dom D$, we let $\cE(F,G) = (\Pi \otimes \nu)(DF DG)$.
Since by \cite[Lemma 3]{LastAnaSto}, the operator $D$ is closed, $\cE$ is a Dirichlet form with domain $\dom \cE = \dom D$.
Moreover, in view of the integration by parts \cref{equation:integration_by_parts_delta}, the generator of $\cE$ is given by $L$.
By \cite[Chapter I Section 3]{BouleauHirsch}, $\CA := \dom D \cap \CL^{\infty}(\Pi)$ is an algebra with respect to the pointwise multiplication; $\dom D$ and $\CA$ are stable by composition with Lipschitz functions; $\CA$ is stable by composition with $\CC^{k}(\dR^{d})$ functions ($k \in \bar{\dN}$).

\subsubsection*{The {carré du champ} operator}
For every $F \in \CA$, we define the \emph{functional carré du champ} of $F$ as the linear form $\Gammaa(F)$ on $\CA$, defined by
\begin{equation*}
  \Gammaa(F)[\Phi] = \cE(F,F \Phi) - \frac{1}{2} \cE(F^{2},\Phi), \quad \text{for all}\ \Phi \in \CA.
\end{equation*}
From~\cite[Proposition I.4.1.1]{BouleauHirsch},
\begin{equation*}
  0 \leq \Gammaa(F)[\Phi] \leq |\Phi|_{\CL^{\infty}(\Pi)} \cE(F), \quad \text{for all}\ F, \Phi \in \CA.
\end{equation*}
This allows us to extend the definition of the linear form $\Gammaa(F)$ to all $F \in \dom \cE$.
For $F \in \dom \cE$, we write that $F \in \dom \Gamma$ if the linear form $\Gammaa(F)$ can be represented by a measure absolutely continuous with respect to $\Pi$; in that case we denote its density by $\Gamma(F)$.
In other words, $F \in \dom \Gamma$ if and only if there exists a non-negative $\Gamma(F) \in \CL^{1}(\Pi)$ such that
\begin{equation*}
  \Gammaa(F)[\Phi] = \esp_{\Pi} \Gamma(F) \Phi,\quad \text{for all}\ \Phi \in \CA.
\end{equation*}
From the general theory, we know that $\dom \Gamma$ is a closed sub-linear space of $\dom \cE$.
In the Poisson case, we prove the following representation of the carré du champ that is a consequence of \cref{lemma:chain_rule_multiplication}.
\begin{proposition}\label{proposition:carre_du_champ}
  We have that $\dom \Gamma = \dom D$ and, for all, $F \in \dom D$:
  \begin{equation*}
    \Gamma(F) = \frac{1}{2} \int {(D_{z}^{+}F)}^{2} \nu(\dd z) + \frac{1}{2} \int {(D_{z}^{-}F)}^{2} \eta(\dd z).
  \end{equation*}
\end{proposition}
We extend $\Gamma$ to a bilinear map
\begin{equation*}
  \Gamma(F,G) = \frac{1}{2} \int D_{z}^{+}F D_{z}^{+}G \nu(\dd z) + \frac{1}{2} \int D_{z}^{+}F D_{z}^{+} G \eta(\dd z), \quad \forall F,G \in \dom D.
\end{equation*}
\begin{remark}
  This representation of $\Gamma$ using the add-one and drop-one operators is, at the formal level, well-known in the literature: it appears without a proof in the seminal paper \cite[191]{BakryEmeryDiffusionsHypercontractives}.
  One of the main assumption of \cite{BakryEmeryDiffusionsHypercontractives} is the existence of an algebra of functions contained in $\dom L$, the so called \emph{standard algebra}.
  In the case of a Poisson point process, it is not clear what to choose for the standard algebra (note that $\CA = \dom \cE \cap \CL^{\infty}(\Pi)$ is \emph{not} included in $\dom L$).
  \cite{DoeblerPeccatiFMT} derives the formula without relying on the notion of standard algebra.
  However, since \cite{DoeblerPeccatiFMT} follows the strategy of \cite{BakryEmeryDiffusionsHypercontractives}, \cite{DoeblerPeccatiFMT} has to assume a restrictive assumption on $F$: $F \in \dom L$ and $F^{2} \in \dom L$.
  In particular, the authors of \cite{DoeblerPeccatiFMT} did not obtain that $\dom \Gamma = \dom \cE$.
  This is why, following \cite{BouleauHirsch}, we use the formalism of Dirichlet forms to compute the carré du champ and obtain a representation for the carré du champ under minimal assumptions.
\end{remark}
\subsubsection*{The energy bracket}
Given two elements $u \in \CL^{2}(\nu \otimes \Pi)$ and $v \in \CL^{2}(\nu \otimes \Pi)$, we define the \emph{energy bracket} of $u$ and $v$: it is the function
\begin{equation*}
  {[u,v]}_{\Gamma}(\eta) = \frac{1}{2} \int u(\eta, z) v(\eta, z) \nu(\dd z) + \frac{1}{2} \int u(\eta - \delta_{z}, z) v(\eta- \delta_{z}, z) \eta(\dd z).
\end{equation*}
The energy bracket can be compared with the two related objects:
\begin{align*}
  & {[u, v]}_{+}(\eta) = \int u(\eta, z) v(\eta, z) \nu(\dd z); \\
  & {[u, v]}_{-}(\eta) = \int u(\eta - \delta_{z}, z) v(\eta - \delta_{z}, z) \eta(\dd z).
\end{align*}
Note that ${[u, v]}_{+}$ is simply the scalar product of $u$ and $v$ in $\CL^{2}(\nu)$.
  By the Cauchy-Schwarz inequality ${[u,v]}_{+} \in \CL^{1}(\Pi)$, and by the Mecke formula:
  \begin{equation*}\label{equation:expectation_energy_bracket}
    \esp_{\Pi} {[u,v]}_{\Gamma} = \esp_{\Pi} {[u,v]}_{+} = \esp_{\Pi} {[u,v]}_{-}.
  \end{equation*}
  If $F$ and $G \in \dom D$, by \cref{proposition:carre_du_champ}, we have that
\begin{equation*}
  \Gamma(F,G) = {[DF, DG]}_{\Gamma}.
\end{equation*}
The fact that the carré du champ is \emph{not} $\nu(DF DG)$ is characteristic of non-local Dirichlet forms.

\subsubsection*{A formula for the divergence}
Since the operator $D$ is not a derivation, \cite[Proposition 1.3.3]{NualartMalliavinCalculus} (obtained in the setting of Malliavin calculus for Gaussian processes) does not hold.
We however have the following Poisson counterpart.
\begin{lemma}
  Let $F \in \dom D$ and $u \in \dom \delta$ such that $Fu \in \dom \delta$.
  Then,
  \begin{equation*}
    \delta(Fu) = F \delta u - [DF, u]_{-}.
  \end{equation*}
\end{lemma}

\begin{proof}
  Let $G \in \CA = \dom D \cap \CL^{\infty}(\Pi)$, and assume moreover that $u \in \CL^{1}(\Pi \otimes \nu)$.
  By integration by parts and the Mecke formula, we find that
  \begin{align*}
    \esp_{\Pi} G \delta(Fu) &= \int F(\eta) u(\eta, z) D_{z}G(\eta)  \nu(\dd z) \Pi(\dd \eta) \\
                      &= \int G(\eta) \int (F(\eta-\delta_{z}) u(\eta-\delta_{z}, z)) \eta(\dd z) \Pi(\dd \eta) - \int G(\eta) \int F(\eta) u(\eta, z) \nu(\dd z) \Pi(\dd \eta).
  \end{align*}
  Using that $F(\eta-\delta_{z})u(\eta-\delta_{z}, z) = F(\eta) u(\eta-\delta_{z}, \eta) - D_{z}^{-}F u(\eta-\delta_{z}, z)$, we conclude by \cref{equation:divergence_pathwise} that
  \begin{equation*}
    \esp_{\Pi} G \delta(Fu) = \esp_{\Pi} G F \delta u - \esp_{\Pi} G [DF, u]_{\eta}.
  \end{equation*}
  We conclude by density.
\end{proof}

\subsubsection*{Algebraic relations for the add and drop operators}
Some immediate algebra yields:
  \begin{align}
    \label{equation:chain_rule_plus_square} D_{z}^{+}F^{2} = 2 F D_{z}^{+}F + {(D_{z}^{+}F)}^{2}; \\
    \label{equation:chain_rule_minus_square} D_{z}^{-}F^{2} = 2 F D_{z}^{-}F - {(D_{z}^{-}F)}^{2}.
  \end{align}

\subsubsection*{An integrated chain rule for the energy}
Recall that we write $\CA$ for the algebra $\dom \cE \cap \CL^{\infty}(\Pi)$.
We now remark that even if $D$ is not a derivation, the Dirichlet energy $\cE$ acts as a derivation.
\begin{lemma}\label{lemma:chain_rule_multiplication}
  Let $F$ and $G \in \CA$, and $u \in \CL^{2}(\Pi \otimes \nu)$.
  Then,
\begin{equation*}
  \esp_{\Pi} {[D(FG), u]}_{\Gamma} = \esp_{\Pi} F {[DG, u]}_{\Gamma} + \esp_{\Pi} G {[DF, u]}_{\Gamma}.
  \end{equation*}
  In particular, with $H \in \dom D$:
  \begin{equation*}\label{equation:energy_derivation}
    \cE(FG,H) = \esp_{\Pi} F {[DG, DH]}_{\Gamma} + \esp_{\Pi} G {[DF, DH]}_{\Gamma}.
  \end{equation*}
  This establishes \cref{proposition:carre_du_champ}.
\end{lemma}
\begin{remark}
  The formula \cref{equation:energy_derivation} for $\cE$ \emph{cannot} be iterated.
  In particular, consistently with the fact that $L$ is not a diffusion, \cref{equation:energy_derivation} does \emph{not} imply $\cE(\phi(F), G) = \esp_{\Pi} \phi'(F) {[DF, DG]}_{\Gamma}$.
\end{remark}
\begin{proof}
  Since $F \in \CL^{\infty}(\Pi)$, by \cref{lemma:D_bounded}, we have that $DF \in \CL^{\infty}(\Pi \otimes \nu)$; and by assumption, $DF \in \CL^{2}(\Pi \otimes \nu)$.
  A similar result holds for $G$, and we find that $DF DG$ is square integrable.
  By the Mecke formula, and \cref{equation:chain_rule_plus_square,equation:chain_rule_minus_square}, we can write:
  \begin{equation*}
    \begin{split}
      \esp_{\Pi} {[D(FG), u]}_{\Gamma} &= \esp_{\Pi} F {[DG, u]}_{\Gamma} + \esp_{\Pi} G {[DF, u]}_{\Gamma} \\
                                       &+ \frac{1}{2} \esp_{\Pi} \int D_{z}^{+}F \otimes DG \otimes u(z) \nu(\dd z) \\
                                       &- \frac{1}{2} \esp_{\Pi} \int (1-D_{z}^{-})F (1-D_{z}^{-})G (1-D_{z}^{-})u(z) \eta(\dd z).
    \end{split}
  \end{equation*}
  By the Mecke formula, the two terms on the two last lines cancel out.
  This proves the first part of the claim.
  To establish \cref{proposition:carre_du_champ}, we simply write, for $F$ and $\Phi \in \CA$:
  \begin{equation*}
    \cE(F, F \Phi) - \frac{1}{2} \cE(F^{2}, \Phi) = \esp_{\Pi} F {[DF, D\Phi]}_{\Gamma} + \esp_{\Pi} \Phi {[DF, DF]}_{\Gamma} - \esp_{\Pi} F {[DF, D\Phi]}_{\Gamma}.
  \end{equation*}
  This shows that $\dom \Gamma \supset \CA$ and that
  \begin{equation*}
    \Gammaa(F)[\Phi] = \esp_{\Pi} {[DF,DF]}_{\Gamma} \Phi.
  \end{equation*}
  We extend this expression to $\dom \cE = \dom D$.
  This concludes the proof.
\end{proof}

\printbibliography%
\end{document}